\documentclass[11pt,twoside,reqno]{amsart}

\usepackage[english]{babel}
\usepackage[a4paper,centering]{geometry}
\usepackage[T1]{fontenc}
\usepackage[utf8]{inputenc}
\usepackage{lmodern}
\usepackage[verbose=silent]{microtype}

\usepackage{amsmath}
\usepackage{amssymb}

\usepackage{graphicx}
\usepackage{caption}
\DeclareCaptionLabelSeparator{period}{. }
\usepackage{esint}
\usepackage{mathrsfs}
\usepackage{tikz}

\PassOptionsToPackage{hyphens}{url}
\usepackage{csquotes}
\usepackage[backend=biber, style=numeric-comp, giveninits=true, sorting=nyt, doi=false, url=false, isbn=false, maxbibnames=99, defernumbers=true, labeldateparts=true]{biblatex}
\renewbibmacro{in:}{}
\DeclareFieldFormat[article,unpublished,incollection,inproceedings,book]{title}{\emph{#1}\isdot}
\DeclareFieldFormat[article]{volume}{\textrm{\mkbibbold{#1}}}
\DeclareFieldFormat[article,incollection,inproceedings,book]{number}{\textrm{\bibstring{number}\nobreakspace#1}}
\DeclareFieldFormat[article,incollection,inproceedings,book]{pages}{\textrm{#1}}
\DeclareFieldFormat[article]{eid}{\textrm{#1}}
\DeclareFieldFormat[article]{journaltitle}{\textrm{#1}}
\DefineBibliographyStrings{english}{number = {no\adddot}}
\renewbibmacro*{volume+number+eid}{%
  \iffieldundef{volume}{}
    {\printfield{volume}\setunit{\addspace}}%
  \csname iffieldequalstr\endcsname{journaltitle}{Int. Math. Res. Not. IMRN}
    {\printdate}
    {\printtext[parens]{\printdate}}%
  \setunit{\addcomma\space}%
  \printfield{number}%
  \setunit{\addcomma\space}%
  \printfield{eid}%
  \setunit{\addcomma\space}%
  \printfield{pages}}
\renewbibmacro*{journal+issuetitle}{%
  \usebibmacro{journal}%
  \setunit*{%
    \csname iffieldequalstr\endcsname{journaltitle}{Int. Math. Res. Not. IMRN}
      {\addspace}
      {\addcomma\space}}%
  \iffieldundef{series}{}
    {\newunit\printfield{series}\setunit{\addspace}}%
  \usebibmacro{volume+number+eid}}
\renewbibmacro*{issue+date}{}
\DeclareBibliographyDriver{article}{%
  \usebibmacro{bibindex}%
  \usebibmacro{begentry}%
  \usebibmacro{author/translator+others}%
  \setunit{\addperiod\space}%
  \usebibmacro{title}%
  \setunit{\addperiod\space}%
  \printlist{language}%
  \newunit\newblock
  \usebibmacro{byauthor}%
  \newunit\newblock
  \usebibmacro{bytranslator+others}%
  \newunit\newblock
  \printfield{version}%
  \newunit\newblock
  \usebibmacro{journal+issuetitle}%
  \newunit\newblock
  \usebibmacro{byeditor+others}%
  \newunit\newblock
  \printfield{note}%
  \newunit\newblock
  \iftoggle{bbx:isbn}{\printfield{issn}}{}%
  \newunit\newblock
  \usebibmacro{doi+eprint+url}%
  \newunit\newblock
  \usebibmacro{addendum+pubstate}%
  \setunit{\bibpagerefpunct}\newblock
  \usebibmacro{pageref}%
  \newunit\newblock
  \iftoggle{bbx:related}{\usebibmacro{related:init}\usebibmacro{related}}{}%
  \usebibmacro{finentry}}
\DeclareBibliographyDriver{unpublished}{%
  \usebibmacro{bibindex}%
  \usebibmacro{begentry}%
  \usebibmacro{author/translator+others}%
  \setunit{\addperiod\space}%
  \usebibmacro{title}%
  \setunit{\addperiod\space}%
  \printfield{note}%
  \setunit{\addcomma\space}%
  \printfield{year}%
  \setunit{\addperiod}%
  \usebibmacro{finentry}}

\usepackage{hyperref}
\hypersetup{
  colorlinks=true,
  linkcolor=black,
  anchorcolor=black,
  citecolor=black,
  filecolor=black,
  menucolor=black,
  runcolor=black,
  urlcolor=black,
  pdfencoding=auto
}
\usepackage{aliascnt}
\usepackage[compress,capitalize,noabbrev,nameinlink]{cleveref}

\theoremstyle{plain}
\newtheorem{theorem}{Theorem}[section]
\newaliascnt{proposition}{theorem}
\newtheorem{proposition}[proposition]{Proposition}
\aliascntresetthe{proposition}
\newaliascnt{lemma}{theorem}
\newtheorem{lemma}[lemma]{Lemma}
\aliascntresetthe{lemma}
\newaliascnt{conjecture}{theorem}
\newtheorem{conjecture}[conjecture]{Conjecture}
\aliascntresetthe{conjecture}
\theoremstyle{remark}
\newtheorem*{declaration}{Declaration}

\newcommand{\pairconjunction}{~and~}
\newcommand{\midconjunction}{,~}
\newcommand{\lastconjunction}{,~and~}
\crefname{equation}{}{}
\Crefname{equation}{}{}
\crefname{theorem}{Theorem}{Theorems}
\Crefname{theorem}{Theorem}{Theorems}
\crefname{proposition}{Proposition}{Propositions}
\Crefname{proposition}{Proposition}{Propositions}
\crefname{lemma}{Lemma}{Lemmas}
\Crefname{lemma}{Lemma}{Lemmas}
\crefname{conjecture}{Conjecture}{Conjectures}
\Crefname{conjecture}{Conjecture}{Conjectures}
\crefname{figure}{Figure}{Figures}
\Crefname{figure}{Figure}{Figures}
\crefname{section}{Section}{Sections}
\Crefname{section}{Section}{Sections}
\numberwithin{equation}{section}
\crefformat{equation}{(#2#1#3)}
\crefformat{theorem}{#2Theorem~#1#3}
\crefformat{conjecture}{#2Conjecture~#1#3}
\crefformat{lemma}{#2Lemma~#1#3}
\crefformat{proposition}{#2Proposition~#1#3}
\crefformat{figure}{#2Figure~#1#3}
\crefformat{section}{#2Section~#1#3}
\Crefformat{equation}{(#2#1#3)}
\Crefformat{theorem}{#2Theorem~#1#3}
\Crefformat{conjecture}{#2Conjecture~#1#3}
\Crefformat{lemma}{#2Lemma~#1#3}
\Crefformat{proposition}{#2Proposition~#1#3}
\Crefformat{figure}{#2Figure~#1#3}
\Crefformat{section}{#2Section~#1#3}
\crefmultiformat{equation}{(#2#1#3)}{\pairconjunction(#2#1#3)}{\midconjunction(#2#1#3)}{\lastconjunction(#2#1#3)}
\crefmultiformat{theorem}{Theorems~#2#1#3}{\pairconjunction#2#1#3}{\midconjunction#2#1#3}{\lastconjunction#2#1#3}
\crefmultiformat{conjecture}{Conjectures~#2#1#3}{\pairconjunction#2#1#3}{\midconjunction#2#1#3}{\lastconjunction#2#1#3}
\crefmultiformat{lemma}{Lemmas~#2#1#3}{\pairconjunction#2#1#3}{\midconjunction#2#1#3}{\lastconjunction#2#1#3}
\crefmultiformat{proposition}{Propositions~#2#1#3}{\pairconjunction#2#1#3}{\midconjunction#2#1#3}{\lastconjunction#2#1#3}
\crefmultiformat{figure}{Figures~#2#1#3}{\pairconjunction#2#1#3}{\midconjunction#2#1#3}{\lastconjunction#2#1#3}
\crefmultiformat{section}{Sections~#2#1#3}{\pairconjunction#2#1#3}{\midconjunction#2#1#3}{\lastconjunction#2#1#3}
\Crefmultiformat{equation}{(#2#1#3)}{\pairconjunction(#2#1#3)}{\midconjunction(#2#1#3)}{\lastconjunction(#2#1#3)}
\Crefmultiformat{theorem}{Theorems~#2#1#3}{\pairconjunction#2#1#3}{\midconjunction#2#1#3}{\lastconjunction#2#1#3}
\Crefmultiformat{conjecture}{Conjectures~#2#1#3}{\pairconjunction#2#1#3}{\midconjunction#2#1#3}{\lastconjunction#2#1#3}
\Crefmultiformat{lemma}{Lemmas~#2#1#3}{\pairconjunction#2#1#3}{\midconjunction#2#1#3}{\lastconjunction#2#1#3}
\Crefmultiformat{proposition}{Propositions~#2#1#3}{\pairconjunction#2#1#3}{\midconjunction#2#1#3}{\lastconjunction#2#1#3}
\Crefmultiformat{figure}{Figures~#2#1#3}{\pairconjunction#2#1#3}{\midconjunction#2#1#3}{\lastconjunction#2#1#3}
\Crefmultiformat{section}{Sections~#2#1#3}{\pairconjunction#2#1#3}{\midconjunction#2#1#3}{\lastconjunction#2#1#3}

\usepackage{xstring}
\makeatletter
\def\setrangeformatparen#1{%
  \kaenmaki@setrangeformatparen\crefrangeformat{#1}%
  \kaenmaki@setrangeformatparen\Crefrangeformat{#1}%
}
\def\kaenmaki@setrangeformatparen#1#2{%
  #1{#2}{%
    \StrBefore{##1}{.}[\seca]%
    \StrBefore{##2}{.}[\secb]%
    \StrBehind{##1}{.}[\numa]%
    \StrBehind{##2}{.}[\numb]%
    \IfInteger{\seca}{%
      \IfInteger{\secb}{%
        \IfInteger{\numa}{%
          \IfInteger{\numb}{%
            \ifnum\seca=\secb\relax
              \ifnum\numb=\numexpr\numa +1 \relax
                (##3##1##4)\pairconjunction(##5##2##6)
              \else
                (##3##1##4)--(##5##2##6)
              \fi
            \else
              (##3##1##4)--(##5##2##6)
            \fi
          }{(##3##1##4)--(##5##2##6)}%
        }{(##3##1##4)--(##5##2##6)}%
      }{(##3##1##4)--(##5##2##6)}%
    }{(##3##1##4)--(##5##2##6)}%
  }%
}
\def\setrangeformatplain#1#2{%
  \kaenmaki@setrangeformatplain\crefrangeformat{#1}{#2}%
  \kaenmaki@setrangeformatplain\Crefrangeformat{#1}{#2}%
}
\def\kaenmaki@setrangeformatplain#1#2#3{%
  #1{#2}{%
    \StrBefore{##1}{.}[\seca]%
    \StrBefore{##2}{.}[\secb]%
    \StrBehind{##1}{.}[\numa]%
    \StrBehind{##2}{.}[\numb]%
    \IfInteger{\seca}{%
      \IfInteger{\secb}{%
        \IfInteger{\numa}{%
          \IfInteger{\numb}{%
            \ifnum\seca=\secb\relax
              \ifnum\numb=\numexpr\numa +1 \relax
                #3##3##1##4\pairconjunction##5##2##6%
              \else
                #3##3##1##4--##5##2##6%
              \fi
            \else
              #3##3##1##4--##5##2##6%
            \fi
          }{#3##3##1##4--##5##2##6}%
        }{#3##3##1##4--##5##2##6}%
      }{#3##3##1##4--##5##2##6}%
    }{#3##3##1##4--##5##2##6}%
  }%
}
\makeatother
\setrangeformatparen{equation}
\setrangeformatplain{theorem}{Theorems~}
\setrangeformatplain{conjecture}{Conjectures~}
\setrangeformatplain{lemma}{Lemmas~}
\setrangeformatplain{proposition}{Propositions~}
\setrangeformatplain{figure}{Figures~}
\setrangeformatplain{section}{Sections~}

\newcommand{\C}{\mathbb{C}}
\newcommand{\N}{\mathbb{N}}
\newcommand{\Q}{\mathbb{Q}}
\newcommand{\R}{\mathbb{R}}
\newcommand{\Z}{\mathbb{Z}}
\newcommand{\eps}{\varepsilon}
\renewcommand{\ge}{\geqslant}
\renewcommand{\le}{\leqslant}

\DeclareMathOperator{\supp}{supp}
\DeclareMathOperator{\rad}{rad}
\DeclareMathOperator{\den}{den}

\allowdisplaybreaks

\begin{document}

\title{Rational points of exact period three for the quadratic family}

\author{Antti K\"aenm\"aki}
\address[Antti K\"aenm\"aki]
{Department of Physics and Mathematics \\
  University of Eastern Finland \\
  P.O.\ Box 111 \\
  FI-80101 Joensuu \\
  Finland}
\email{antti@kaenmaki.net}

\author{Pekka Kosunen}
\address[Pekka Kosunen]
{Department of Physics and Mathematics \\
  University of Eastern Finland \\
  P.O.\ Box 111 \\
  FI-80101 Joensuu \\
  Finland}
\email{pekka.kosunen@uef.fi}

\date{\today}

\begin{abstract}
  For the quadratic family $f_c(x) = x^2 + c$ over $\Q$, we isolate a general $S$-unit and ABC template for rational periodic cycles: under the ABC conjecture, any rational quantity built from the consecutive-point sums of a rational period-$n$ cycle whose value and shift-by-$1$ are both $S$-units (with $S$ the support of the denominator of $c$) has logarithmic height at most $(1+\eps)\sum_{p \in S} \log p + O_\eps(1)$. Specializing to period $3$ in the change of coordinates $u = x_0 + x_1$ and $v = x_0 + x_2$, the period-$3$ condition reduces to $v = -1/(1+u)$, forcing the parameter $t = u$ to satisfy $t, t+1 \in \Z[S^{-1}]^\times$; the template then yields, under ABC, $\max\{|a|,|b|,|a+b|\} \le C_\eps \rad(d)^{1+\eps}$ for $t = a/b$ in lowest terms. We also characterize the period-$3$ locus on the $(u,v)$-plane intrinsically, isolating the ghost curve. Extending the height bound to higher periods remains an open structural question.
\end{abstract}

\maketitle

\section{Introduction}

The quadratic family
\begin{equation*}
  f_c(x) = x^2 + c
\end{equation*}
with parameter $c \in \Q$ is a central object in arithmetic dynamics; see Silverman \cite{silverman2007} for a comprehensive treatment. A central problem is to determine which exact periods occur for rational points and, for each occurring period, to describe all rational cycles. In this paper, we focus on exact period $3$, where rational cycles exist and admit explicit parametrizations. We ask what arithmetic restrictions on the period parameter are forced by the denominator of $c$.

For low periods, the picture is well understood. Fixed points of $f_c$ satisfy a quadratic equation, and $2$-cycles are parametrized by a rational curve, so both cases are elementary. Morton \cite[Theorem~4]{morton1998} proved that no rational $4$-cycle exists, Flynn--Poonen--Schaefer \cite[Theorem~1]{flynn-poonen-schaefer1997} excluded period $5$, and Stoll \cite[Theorem~7]{stoll2008} conditionally excluded period $6$, assuming the Birch--Swinnerton-Dyer conjecture. These results support the Flynn--Poonen--Schaefer conjecture, which asserts that for every $c \in \Q$, the map $f_c$ has no rational cycle of exact period $n \ge 4$; see \cite[Conjecture~2]{flynn-poonen-schaefer1997} for the original formulation and Poonen \cite[Theorems~1 and 3]{poonen1998} for a refined version that also predicts the full preperiodic structure.

Period $3$ is the largest period expected to occur, and rational $3$-cycles exist for suitable parameters. The period-$3$ dynatomic curve, the curve whose points are the pairs $(x,c)$ with $x$ of exact period $3$ for $f_c$, has genus $0$, meaning that its points can be described by a single rational parameter, so there is a complete one-parameter family of rational $3$-cycles; explicit formulas appear in Poonen \cite[Theorem~1(3)]{poonen1998} and are recorded by Fu--Stoll \cite[Theorem~2.1(1)]{fu-stoll2024}. The arithmetic of these cycles has been studied in detail by Morton--Raianu \cite[Theorems~2(a) and 4]{mortonraianu2022}, who computed the exact denominator of the period parameter and used the resulting structure to prove a minimality property for $c = -29/16$. A result of Walde--Russo \cite{walde-russo1994}, refined by Eliahou--Fares \cite[Proposition~2.1]{eliahou-fares2022}, shows that the denominator of any periodic point is determined by the denominator of $c$. The question we address is what additional structure this denominator rigidity produces when specialized to period $3$, and what quantitative consequences it has.

The key technical device is a change of coordinates
\begin{equation*}
  u = x_0 + x_1 \qquad \text{and} \qquad v = x_0 + x_2
\end{equation*}
introduced for quadratic iteration by Erkama \cite[\S 3--4]{erkama2006}, who used $u = x_0+x_1$ to uniformize the period-$3$ curve and $v = x_0+x_2$ to uniformize the period-$4$ curve; Kosunen \cite[\S2]{kosunen2017} later combined the two into a single two-dimensional $(u,v)$-plane carrying an explicit quadratic iteration map. In the $(u,v)$-model, the period-$3$ condition collapses to the single relation $v = -1/(1+u)$, in place of the degree-$6$ dynatomic polynomial in the original variables. This simplification makes the $S$-unit structure transparent. By periodicity, the product of consecutive-point sums equals $1$. By the denominator theorem of Walde--Russo, the denominator $\den(c)$ of $c$ is a perfect square $d^2$, and every prime factor of $t = x_0+x_1$ and of $t+1$ lies in the set $S$ of primes dividing $d$. Applying the ABC conjecture, \cref{conj:abc}, to the resulting coprime triples then bounds the logarithmic height of $t$, namely $h(t) = \log\max\{|a|,|b|\}$ for $t = a/b$ in lowest terms, in terms of the primes in $S$; this bound is made precise in \cref{thm:period3}, and the same $S$-unit structure also yields unconditional bounds, recorded in \cref{prop:period3-evertse-count,prop:period3-unconditional-height}.

Our contributions are threefold. First, \cref{thm:abc-template} isolates a general $S$-unit and ABC template for periodic orbits of $f_c$, valid for every exact period $n \ge 2$: under \cref{conj:abc}, any rational quantity $\tau_n$ built from the consecutive-point sums $u_0,\ldots,u_{n-1}$ of a rational period-$n$ cycle, for which both $\tau_n$ and $\tau_n+1$ are $S$-units, satisfies $h(\tau_n) \le (1+\eps)\sum_{p\in S}\log p + O_\eps(1)$. The ABC conjecture has been applied in arithmetic dynamics before, for instance to primitive prime divisors in orbits by Gratton, Nguyen, and Tucker \cite[Theorem~1.1]{gratton-nguyen-tucker2013}, but we are not aware of a prior application bounding the height of rational periodic cycles of $f_c$. Second, specializing the template to period $3$ with $\tau_3 = t$ gives our main quantitative result \cref{thm:period3}, tying the complexity of a rational period-$3$ cycle to the prime factorization of $\den(c)$. Third, \cref{lem:period3-uv-direct} characterizes the period-$3$ locus intrinsically, within the induced $(u,v)$-dynamics and with no reference to $f_c$: a rational point $(u,v) \in \Q^2$ has exact period $3$ if and only if $u \in \Q \setminus \{0,-1\}$ and $v = -1/(1+u)$. Its proof isolates the \emph{ghost curve}, the second algebraic component of $u_0 u_1 u_2 = 1$ that satisfies the algebra but carries no rational period-$3$ orbit. 

The period-$3$ specialization rests on an unconditional $S$-unit constraint: \cref{prop:period3-rational-sunit} recasts the denominator rigidity of Eliahou--Fares \cite[Corollary~2.6]{eliahou-fares2022} and Morton--Raianu \cite[Theorem~2(a)]{mortonraianu2022} as the $S$-unit equation $t/(t+1)+1/(t+1) = 1$, which is what makes \cref{conj:abc} applicable. Feeding the same constraint into classical Diophantine bounds in place of \cref{conj:abc} yields two unconditional companions: a count of admissible parameters $c$ via Evertse's $S$-unit bound \cite[Theorem~1]{evertse1984} in \cref{prop:period3-evertse-count}, and a weaker height bound via the Stewart--Yu form of ABC \cite[Theorem~1]{stewart-yu2001} in \cref{prop:period3-unconditional-height}.

The genuinely new input is this conditional height bound, its template form, and the intrinsic characterization isolating the ghost curve. The denominator rigidity and $S$-unit structure that feed the template go back to Walde--Russo, Eliahou--Fares, and Morton--Raianu, while the $(u,v)$-model, the period-$3$ relation, and the parametrization go back to Erkama, Kosunen, and Poonen. For these classical ingredients we give shorter, self-contained proofs, internal to the $(u,v)$-model and valid over an arbitrary field of characteristic zero: the telescoping identity of Erkama in \cref{lem:telescoping}, the period-$3$ relation $v = -1/(1+u)$ in \cref{lem:period3-uv}, and the rational period-$3$ parametrization in \cref{prop:period3-classification}.

\begin{declaration}
  In preparing this paper the authors used large language models for language editing throughout, and in some cases to draft English prose from the authors' own descriptions of the intended mathematical content. The models were also used as an aid in checking proofs for correctness and to help locate references. All results and arguments are the authors' own and have been verified independently of the models, and every reference has been checked by the authors against the original source. The authors take full responsibility for the content of this paper.
\end{declaration}

\section{Main theorem and prior input} \label{sec:main}

In this section, we state the main theorem in the original $(x,c)$-coordinates; the detailed ingredients are proved in later sections. We first fix notation. For $c \in \Q$, write $\den(c)$ for the denominator of $c$ in lowest terms. For each nonzero integer $d$, write $\supp(d)$ for the set of prime divisors of $d$. For a finite set $S$ of primes, write $\Z[S^{-1}]^\times$ for the group of nonzero rationals whose numerator and denominator are divisible only by primes in $S$; its elements are called \emph{$S$-units}. For each nonzero integer $d$, write $\rad(d) = \prod_{p \in \supp(d)} p$ for the product of the distinct prime divisors of $d$. For a rational number $t = a/b$ in lowest terms, write $h(t) = \log\max\{|a|,|b|\}$ for the logarithmic height of $t$. We next recall the ABC conjecture over $\Q$, formulated independently by Masser \cite{masser1985} and Oesterl\'e \cite{oesterle1988}.

\begin{conjecture} \label{conj:abc}
  For every $\eps > 0$ there exists $C_\eps > 0$ such that $|c| \le C_\eps \rad(abc)^{1+\eps}$ for all coprime integers $a,b,c$ with $a+b = c$.
\end{conjecture}

A proof of \cref{conj:abc} has been announced by Mochizuki \cite{mochizuki2021iut4}, but it remains contested; see Scholze and Stix \cite{scholze-stix2018}. We treat \cref{conj:abc} as a conjecture throughout.

The main theorem is the period-$3$ specialization of \cref{thm:abc-template}.

\begin{theorem} \label{thm:period3}
  Assume \cref{conj:abc}. Let $c \in \Q$ admit a rational point $x_0$ of exact period $3$ for $f_c(x) = x^2 + c$, write $\den(c) = d^2$ with $d \in \N$, set $S = \supp(d)$, let $x_1 = f_c(x_0)$, and set $t = x_0+x_1$. If $t = a/b$ in lowest terms, then for every $\eps > 0$ there is $C_\eps > 0$ such that
  \begin{equation} \label{eq:period3-abc}
    \max\{|a|,|b|,|a+b|\} \le C_\eps \rad(d)^{1+\eps}.
  \end{equation}
  Equivalently,
  \begin{equation} \label{eq:abc-log-height-bound-t}
    h(t) \le (1+\eps)\sum_{p\in S}\log p + O_\eps(1).
  \end{equation}
\end{theorem}

We prove \cref{thm:period3} in \cref{sec:period3-arith}, by specializing the template \cref{thm:abc-template} to the $S$-unit constraint of \cref{prop:period3-rational-sunit}. The main theorem is conditional on \cref{conj:abc}. Two unconditional companions, both proved in \cref{sec:period3-arith}, record what the same $S$-unit structure yields without it: the period parameter $t$ and its shift $t+1$ are forced to be $S$-units, exactly as in the proof of \cref{thm:period3}, and feeding this constraint into classical Diophantine bounds in place of \cref{conj:abc} gives the following two statements. The first replaces the height bound by a count.

\begin{proposition} \label{prop:period3-evertse-count}
  Let $S$ be a finite set of primes. For $c \in \Q$ such that $f_c$ admits a rational cycle of exact period $3$, write $\den(c) = d^2$ with $d \in \N$. The number of such $c$ with $\supp(d) \subseteq S$ is at most $3 \cdot 7^{2\# S + 3}$.
\end{proposition}

Where \cref{thm:period3} bounds the height of a single admissible parameter $t$, \cref{prop:period3-evertse-count} bounds the number of admissible parameters $c$ outright, by a quantity depending only on $\# S$ and not on their height, through Evertse's bound \cite[Theorem~1]{evertse1984} on the number of solutions of an $S$-unit equation. The two bounds apply different Diophantine machinery to the common $S$-unit equation of \cref{prop:period3-rational-sunit}: \cref{thm:period3} invokes \cref{conj:abc} to bound a height, whereas \cref{prop:period3-evertse-count} invokes Evertse's theorem to bound a count. Neither is a consequence of the other. In particular, \cref{prop:period3-evertse-count} does not follow from \cref{thm:period3}: the height bound controls only the size of each admissible $t$, and even under \cref{conj:abc} it yields only a count depending on the actual primes in $S$, not a bound uniform in $\# S$. The second companion keeps the height-bound form of \cref{thm:period3} but removes the dependence on \cref{conj:abc}.

\begin{proposition} \label{prop:period3-unconditional-height}
  Let $c \in \Q$ admit a rational point $x_0$ of exact period $3$ for $f_c$, set $x_1 = f_c(x_0)$ and $t = x_0+x_1$, write $\den(c) = d^2$ with $d \in \N$, and set $S = \supp(d)$. Then there is a constant $C_0 > 0$ such that
  \begin{equation} \label{eq:stewart-yu-height-bound-t}
    h(t) < C_0 \biggl( \prod_{p \in S} p \biggr)^{1/3} \biggl(\sum_{p \in S} \log p \biggr)^3.
  \end{equation}
\end{proposition}

The bound of \cref{prop:period3-unconditional-height} is unconditional, resting on the Stewart--Yu form of the ABC inequality \cite[Theorem~1]{stewart-yu2001} in place of \cref{conj:abc}. It is correspondingly weaker than \cref{thm:period3}: under \cref{conj:abc}, \cref{thm:period3} bounds $h(t)$ by $(1+\eps)\log \prod_{p \in S} p + O_\eps(1)$, whereas \cref{prop:period3-unconditional-height} only bounds it by a positive power of $\prod_{p \in S} p$. Thus \cref{thm:period3} supersedes \cref{prop:period3-unconditional-height} whenever \cref{conj:abc} is available, and the companion records the strongest height bound known without it.

The paper is organized as follows. In \cref{sec:periodn}, we collect results valid for arbitrary period $n \ge 2$: the telescoping identity, the Eliahou--Fares denominator transfer, the resulting $S$-unit structure of consecutive-point sums, and the conditional ABC height-bound template \cref{thm:abc-template}. The denominator transfer \cref{prop:ef-transfer} goes back to Walde--Russo \cite{walde-russo1994} and was further developed by Eliahou--Fares \cite{eliahou-fares2022}, who also proved the Flynn--Poonen--Schaefer conjecture for parameters whose denominator has at most two distinct prime factors; we build on their $S$-unit conclusion to obtain \cref{thm:abc-template}. In \cref{sec:uv-model}, we introduce the $(u,v)$-model of Kosunen \cite{kosunen2017} and prove the birational dictionary, recast for an arbitrary field of characteristic zero. In \cref{sec:period3}, we specialize to period $3$. The relation $v = -1/(1+u)$ in \cref{lem:period3-uv} was derived by Erkama \cite[\S3]{erkama2006} from his period-$3$ condition $(x_0+x_2)(1+(x_0+x_1)) = -1$, and the rational parametrization \cref{prop:period3-classification} appears in equivalent form in Poonen \cite{poonen1998} and Fu--Stoll \cite{fu-stoll2024} with an independent derivation in Erkama \cite{erkama2006}; we give shorter proofs internal to the $(u,v)$-model. In \cref{sec:period3-ghost}, we then characterize the period-$3$ locus intrinsically in \cref{lem:period3-uv-direct}, isolating the ghost curve; the underlying factorization is visible in Kosunen \cite{kosunen2017} but is not extracted there as a separate result. In \cref{sec:period3-arith}, we record the unconditional cubic bound \cref{prop:fixed-z-period3-cubic} and the $S$-unit constraint \cref{prop:period3-rational-sunit}, prove \cref{thm:period3} by specializing \cref{thm:abc-template} to that constraint, and prove the unconditional companions \cref{prop:period3-evertse-count,prop:period3-unconditional-height}. \cref{prop:period3-rational-sunit} is logically equivalent to the exact-denominator formula $\den(c) = 16(mn(m+n)/2)^2$ of Morton--Raianu \cite{mortonraianu2022} for $t = m/n$ in lowest terms; the $S$-unit form is what makes \cref{conj:abc} applicable and yields the explicit height bound \cref{thm:period3}.

\section{Telescoping and $S$-unit structure} \label{sec:periodn}

Although this paper focuses on period $3$, the denominator-transfer and telescoping parts of the argument extend to all exact periods $n \ge 2$. We collect these general ingredients here. The starting point is a telescoping identity: \cref{lem:telescoping} shows that for any exact cycle of $f_c$, the product of the consecutive-point sums $u_j = x_j+x_{j+1}$ equals $1$. This identity is purely algebraic and holds over any field of characteristic zero. The denominator theorem of Walde--Russo, recorded as \cref{prop:ef-transfer}, then combines with the telescoping identity to yield \cref{prop:periodn-sunit}: each $u_j$ is an $S$-unit. Under \cref{conj:abc}, the $S$-unit structure yields a height bound for any rational quantity $\tau_n$ of the $u_j$ for which both $\tau_n$ and $\tau_n+1$ remain $S$-units; this is \cref{thm:abc-template}.

The product identity in the following lemma is due to Erkama \cite[(5.1)]{erkama2006}; we record a self-contained proof over an arbitrary field of characteristic zero. Recall that a field $K$ has characteristic $0$ if $n\cdot 1_K \neq 0$ for every integer $n \ge 1$. 

\begin{lemma} \label{lem:telescoping}
  Let $K$ be a field of characteristic $0$ and let $x_{j+1} = x_j^2+c$ with $c \in K$. If $x_0$ has exact period $n \ge 2$, then
  \begin{equation*}
    \prod_{j\in\{0,\ldots,n-1\}} (x_j+x_{j+1}) = 1.
  \end{equation*}
\end{lemma}

\begin{proof}
  Since $x_{k+1}-x_{\ell+1} = x_k^2-x_\ell^2 = (x_k-x_\ell)(x_k+x_\ell)$, we have
  \begin{equation} \label{eq:basic-telescope}
    x_k+x_\ell = \frac{x_{k+1}-x_{\ell+1}}{x_k-x_\ell}
  \end{equation}
  whenever $x_k \neq x_\ell$. Assume that $x_0$ has exact period $n \ge 2$. Then $x_n = x_0$ and $x_{n+1} = x_1$, and the points $x_0,\ldots,x_{n-1}$ are pairwise distinct. In particular, $x_j \neq x_{j+1}$ for every $j \in \{0,\ldots,n-1\}$. Apply \cref{eq:basic-telescope} with $(k,\ell) = (0,1), (1,2), \ldots, (n-1,n)$ and multiply. This gives
  \begin{equation*}
    \prod_{j\in\{0,\ldots,n-1\}} (x_j+x_{j+1}) = \prod_{j\in\{0,\ldots,n-1\}} \frac{x_{j+1}-x_{j+2}}{x_j-x_{j+1}} = \frac{x_n-x_{n+1}}{x_0-x_1} = 1.
  \end{equation*}
  as claimed.
\end{proof}

The second ingredient is arithmetic. The following denominator theorem, due to Walde--Russo \cite{walde-russo1994}, determines the denominator of every periodic point from the denominator of $c$.

\begin{proposition} \label{prop:ef-transfer}
  Let $c \in \Q$ and assume that $f_c$ has a rational periodic point. Then $\den(c)$ is a perfect square; writing $\den(c) = d^2$ with $d \in \N$, every rational periodic point of $f_c$ has denominator $d$.
\end{proposition}

\begin{proof}
  We record the $p$-adic argument of Eliahou--Fares \cite[Proposition~2.1]{eliahou-fares2022}. Recall that, for a prime $p$, the \emph{$p$-adic valuation} $v_p(x)$ of a nonzero $x \in \Q$ is the unique integer $r$ such that $x = p^r x_1/x_2$ with $x_1, x_2$ integers coprime to $p$, and $v_p(0) = \infty$; it satisfies $v_p(yz) = v_p(y) + v_p(z)$ and $v_p(y+z) \ge \min\{v_p(y),v_p(z)\}$, with equality when $v_p(y) \neq v_p(z)$.

  Fix a prime $p$ and let $x \in \Q$. Suppose $v_p(c) \ge 0$ and $v_p(x) < 0$. Then $v_p(x^2) = 2v_p(x) < v_p(c)$, so $v_p(f_c(x)) = 2v_p(x)$, and iterating gives $v_p(f_c^n(x)) = 2^n v_p(x) \to -\infty$. Suppose instead $v_p(c) < 0$ and $v_p(x) \neq v_p(c)/2$. If $v_p(x) < v_p(c)/2$, then $v_p(x^2) = 2v_p(x) < v_p(c)$, so $v_p(f_c(x)) = 2v_p(x) < v_p(c)/2$ and the same iteration drives the orbit to $-\infty$. If $v_p(x) > v_p(c)/2$, then $v_p(x^2) > v_p(c)$, so $v_p(f_c(x)) = v_p(c) < v_p(c)/2$, reducing to the previous case. Hence, whenever the $f_c$-orbit of $x$ is bounded $p$-adically, we have $v_p(x) \ge 0$ if $v_p(c) \ge 0$ and $v_p(x) = v_p(c)/2$ if $v_p(c) < 0$.

  If $f_c$ admits a rational periodic point $x$, then the orbit of $x$ is finite, so the previous dichotomy applies for every prime $p$. For each prime with $v_p(c) < 0$, it forces $v_p(x) = v_p(c)/2$, and hence $v_p(c)$ is even. Thus $\den(c) = d^2$ for some $d \in \N$ with $v_p(d) = -v_p(c)/2$ for $p \in \supp(d)$.

  Now let $y$ be any rational periodic point of $f_c$. Applying the same dichotomy to $y$ gives $v_p(y) = -v_p(d)$ for $p \in \supp(d)$ and $v_p(y) \ge 0$ for $p \notin \supp(d)$. Hence $\den(y) = d$.
\end{proof}

Combining \cref{lem:telescoping} with \cref{prop:ef-transfer} gives the $S$-unit structure of the consecutive-point sums.

\begin{proposition} \label{prop:periodn-sunit}
  Let $n \in \N$ with $n \ge 2$, let $c \in \Q$, and assume that $f_c$ has a rational cycle $(x_0,\ldots,x_{n-1})$ of exact period $n$. Write $x_n=x_0$, $\den(c) = d^2$, and $S = \supp(d)$. Then
  \begin{equation} \label{eq:periodn-product}
    \prod_{j\in\{0,\ldots,n-1\}} (x_j+x_{j+1}) = 1
  \end{equation}
  and
  \begin{equation} \label{eq:periodn-sunits}
    x_j+x_{j+1} \in \Z[S^{-1}]^\times
  \end{equation}
  for every $j \in \{0,\ldots,n-1\}$.
\end{proposition}

\begin{proof}
  This is the $S$-unit reformulation of \cite[Proposition~2.5 and Corollary~2.6(i)]{eliahou-fares2022}; we record the short argument. By \cref{prop:ef-transfer}, each $x_j = X_j/d$ with $X_j \in \Z$ and $\gcd(X_j,d) = 1$; set $X_n = X_0$. Then \cref{lem:telescoping} gives \cref{eq:periodn-product}. Substituting $x_j = X_j/d$ yields
  \begin{equation*}
    \prod_{j\in\{0,\ldots,n-1\}} (X_j+X_{j+1}) = d^n,
  \end{equation*}
  so each $X_j+X_{j+1}$ is nonzero, and every prime divisor of $X_j+X_{j+1}$ lies in $S = \supp(d)$. Both the numerator and the denominator of $x_j+x_{j+1} = (X_j+X_{j+1})/d$ have all prime factors in $S$, which gives \cref{eq:periodn-sunits}.
\end{proof}

The previous proposition is unconditional. The next result is conditional on \cref{conj:abc} and provides a height bound template. Given any rational quantity $\tau_n$ built from the consecutive-point sums $u_0,\ldots,u_{n-1}$ such that both $\tau_n$ and $\tau_n+1$ are $S$-units, the ABC conjecture bounds the height of $\tau_n$ in terms of $\rad(d)$. For period $3$, the natural choice is $\tau_3 = t$, and the hypothesis is verified by \cref{prop:period3-rational-sunit}.

\begin{theorem} \label{thm:abc-template}
  Assume \cref{conj:abc}. Let $n \in \N$ with $n \ge 2$, let $c \in \Q$, and assume that $f_c$ has a rational cycle of exact period $n$. Write $\den(c) = d^2$ and set $S = \supp(d)$. Suppose that $\tau_n \in \Q$ satisfies
  \begin{equation} \label{eq:periodn-tau-sunits}
    \tau_n,\tau_n+1 \in \Z[S^{-1}]^\times.
  \end{equation}
  If $\tau_n = a/b$ in lowest terms, then for every $\eps > 0$ there exists $C_\eps > 0$ such that
  \begin{equation} \label{eq:periodn-abc-height-bound}
    \max\{|a|,|b|,|a+b|\} \le C_\eps \rad(d)^{1+\eps}.
  \end{equation}
  Equivalently,
  \begin{equation} \label{eq:periodn-abc-log-height-bound}
    h(\tau_n) \le (1+\eps)\sum_{p\in S}\log p + O_\eps(1).
  \end{equation}
\end{theorem}

\begin{proof}
  By \cref{eq:periodn-tau-sunits}, both $\tau_n$ and $\tau_n+1$ are $S$-units, so they are nonzero. Writing $\tau_n = a/b$ in lowest terms gives $\gcd(a,b) = 1$ and $\tau_n+1 = (a+b)/b$, with $\gcd(a+b,b) = \gcd(a,b) = 1$. Hence $a$, $b$, and $a+b$ are nonzero, and all prime divisors of $a$, $b$, and $a+b$ belong to $S$. Therefore
  \begin{equation*}
    \rad(ab(a+b)) \le \rad(d).
  \end{equation*}
  Moreover,
  \begin{equation*}
    \gcd(a,b) = \gcd(-a,a+b) = \gcd(-b,a+b) = 1.
  \end{equation*}
  Thus the identities
  \begin{equation*}
    a+b = a+b, \qquad -a + (a+b) = b, \qquad -b + (a+b) = a
  \end{equation*}
  are all of the form $x+y = z$ with coprime integers. Applying \cref{conj:abc} gives
  \begin{equation*}
    \max\{|a|,|b|,|a+b|\} \le C_\eps \rad(ab(a+b))^{1+\eps} \le C_\eps \rad(d)^{1+\eps}.
  \end{equation*}
  This is \cref{eq:periodn-abc-height-bound}. Taking logarithms and using $h(\tau_n) = \log \max\{|a|,|b|\}$ gives \cref{eq:periodn-abc-log-height-bound}. Conversely, \cref{eq:periodn-abc-log-height-bound} implies \cref{eq:periodn-abc-height-bound}, after increasing the constant, since $|a+b| \le 2\max\{|a|,|b|\}$.
\end{proof}

\section{The \texorpdfstring{$(u,v)$}{u,v}-model and birational dictionary}\label{sec:uv-model}

We now pass from the original orbit coordinates to the $(u,v)$-coordinates. The purpose of this section is to record the birational dictionary between rational orbits of $f_c$ and the rational map $G$ on the $(u,v)$-plane, including the proviso needed to transport exact periods.

Let $c \in \Q$ and $x_0 \in \Q$, and let $(x_n)_{n \ge 0}$ be the orbit defined recursively by
\begin{equation*}
  x_{n+1} = f_c(x_n).
\end{equation*}
Define
\begin{equation} \label{eq:uv-def}
  u_n = x_n + x_{n+1} \qquad \text{and} \qquad v_n = x_n + x_{n+2}.
\end{equation}
The first coordinate records sums of consecutive points. The second coordinate records one-step skipped sums. Period equations for $f_c$ can be transported to these coordinates. Define the rational map $G \colon \C^2 \to \C^2$ by
\begin{equation} \label{eq:G}
  G(u,v) = \biggl( \frac{-u+v+uv}{u}, \frac{u^2-u+v-u^2v-uv+uv^2+v^2}{u} \biggr).
\end{equation}
This formula is defined when $u \neq 0$. The next theorem is the dictionary that allows us to move between the two coordinate systems. The substitution $u = x_0+x_1$ and $v = x_0+x_2$ originates with Kosunen \cite[Theorem~2.1, (2.4), and (2.5)]{kosunen2017}, which gives \cref{eq:x0-in-uv,eq:x1-in-uv,eq:c-in-uv} and the explicit formula for the map $G$; the underlying two-dimensional model on the $(x_0,x_1)$-plane is due to Erkama \cite[(1.1)]{erkama2006}. The proof below collects these formulas in one place and adds the period-equivalence statement (with the $u_k \neq 0$ proviso), which is used implicitly throughout \cite{kosunen2017} but not isolated as a single result there.

\begin{theorem} \label{thm:dictionary}
  Let $c \in \Q$ and $x_0 \in \Q$. Define $u_n$ and $v_n$ by \cref{eq:uv-def}. In particular, for $n = 0$ we have $u_0 = x_0+x_1$ and $v_0 = x_0+x_2$. Set $u = u_0$ and $v = v_0$, and assume that $u \neq 0$. Then
  \begin{align} \label{eq:x0-in-uv}
    x_0 &= \frac{u^2+u-v}{2u}, \\
    x_1 &= \frac{u^2-u+v}{2u}, \label{eq:x1-in-uv} \\
    c &= \frac{-u^4+u^2(2v-3)-v^2+4uv}{4u^2}. \label{eq:c-in-uv}
  \end{align}
  Moreover, the consecutive-point sums obey the recursion $(u_{n+1},v_{n+1}) = G(u_n,v_n)$ for every $n \ge 0$ with $u_n \neq 0$; since $(u_0,v_0) = (u,v)$, this gives $(u_n,v_n) = G^n(u,v)$ whenever $u_0,\ldots,u_{n-1} \neq 0$. In particular, for every $n \ge 1$, $x_0$ has exact period $n$ for $f_c$ if and only if $(u,v)$ has exact period $n$ for $G$ and the first coordinates $u_0,\ldots,u_{n-1}$ on the cycle are nonzero.
\end{theorem}

\begin{proof}
  Set
  \begin{equation} \label{eq:uv-seq}
    u_n = x_n + x_{n+1} \qquad \text{and} \qquad v_n = x_n + x_{n+2}.
  \end{equation}
  Since $x_{n+2}-x_{n+1} = x_{n+1}^2-x_n^2$, we have
  \begin{equation*}
    x_{n+2}-x_{n+1} = (x_{n+1}-x_n)(x_n+x_{n+1}) = (x_{n+1}-x_n)u_n.
  \end{equation*}
  Using $v_n = x_n+x_{n+2}$, this gives
  \begin{equation*}
    v_n = u_n+u_n(x_{n+1}-x_n) = u_n+u_n(u_n-2x_n) = u_n^2+u_n-2u_nx_n.
  \end{equation*}
  If $u_n \neq 0$, then
  \begin{equation} \label{eq:xn-from-uvn}
    x_n = \frac{u_n^2+u_n-v_n}{2u_n} \qquad \text{and} \qquad x_{n+1} = \frac{u_n^2-u_n+v_n}{2u_n}.
  \end{equation}
  For $n = 0$, this gives \cref{eq:x0-in-uv,eq:x1-in-uv}. The formula \cref{eq:c-in-uv} follows from $c = x_1-x_0^2$.

  Assume now $u_n \neq 0$. From \cref{eq:uv-seq,eq:xn-from-uvn} we obtain
  \begin{equation} \label{eq:un1}
    u_{n+1} = x_{n+1}+x_{n+2} = u_n+v_n-2x_n = \frac{-u_n+v_n+u_nv_n}{u_n}.
  \end{equation}
  For the second coordinate, write $x_{n+2} = v_n-x_n$ and $c = x_{n+1}-x_n^2$. Then
  \begin{equation*}
    x_{n+3} = x_{n+2}^2+c = (v_n-x_n)^2+x_{n+1}-x_n^2 = v_n^2-2v_nx_n+x_{n+1}.
  \end{equation*}
  Therefore
  \begin{equation} \label{eq:vn1}
    \begin{split}
      v_{n+1} &= x_{n+1}+x_{n+3} = v_n^2-2v_nx_n+2x_{n+1} \\
      &= \frac{u_n^2-u_n+v_n-u_n^2v_n-u_nv_n+u_nv_n^2+v_n^2}{u_n}.
    \end{split}
  \end{equation}
  Equations \cref{eq:un1,eq:vn1} are exactly the map in \cref{eq:G}. Hence $(u_{n+1},v_{n+1}) = G(u_n,v_n)$ whenever $u_n \neq 0$, and since $(u_0,v_0) = (u,v)$, induction gives $(u_n,v_n) = G^n(u,v)$ for every $n \ge 0$ with $u_0,\ldots,u_{n-1} \neq 0$.

  We first record the period-dividing equivalence. If $x_m = x_0$ for some $m \ge 1$, then $x_{m+1} = x_1$ and $x_{m+2} = x_2$, so $(u_m,v_m) = (u_0,v_0)$. Conversely, if $(u_m,v_m) = (u_0,v_0)$ and $u_m \neq 0$, then \cref{eq:xn-from-uvn} gives $x_m = (u_m^2+u_m-v_m)/(2u_m) = (u_0^2+u_0-v_0)/(2u_0) = x_0$. If $x_0$ has exact period $n$, then the first coordinates needed to iterate $G$ are nonzero: for $n \ge 2$, the equality $u_k = 0$ would give $x_{k+1} = -x_k$ and hence $x_{k+2} = x_{k+1}$, contradicting exact periodicity, while for $n = 1$ this is the standing assumption $u_0 \neq 0$. The period-dividing equivalence now shows that no positive return time for $x_0$ differs from a positive return time for $(u,v)$ under $G$, so the two exact periods are equal.
\end{proof}

\section{Period $3$ in the $(u,v)$-model}\label{sec:period3}

We specialize the $(u,v)$-model to exact period $3$. The telescoping identity of \cref{lem:telescoping} then collapses to the single relation $v = -1/(1+u)$ of \cref{lem:period3-uv}; this relation appears as Kosunen \cite[(4.3)]{kosunen2017}, where it is derived from the full polynomial system $G^3 = \mathrm{id}$, where $G$ is defined in \cref{eq:G}. We give a shorter proof from \cref{lem:telescoping} alone, valid over any field of characteristic $0$. From this relation we obtain the explicit parametrization of all rational period-$3$ cycles in \cref{prop:period3-classification}. 

\begin{lemma} \label{lem:period3-uv}
  Let $K$ be a field of characteristic $0$, let $c \in K$, and assume that $x_0 \in K$ has exact period $3$ for $f_c$. Set
  \begin{equation*}
    u = x_0+x_1 \qquad \text{and} \qquad v = x_0+x_2.
  \end{equation*}
  Then $u \neq 0$, $u \neq -1$, and
  \begin{equation} \label{eq:period3-uv-identity}
    v = -\frac{1}{1+u}.
  \end{equation}
\end{lemma}

\begin{proof}
  If $u = 0$, then $x_1 = -x_0$ and $x_2 = x_1^2+c = x_0^2+c = x_1$, contradicting exact period $3$. If $u = -1$, then $x_1 = -1-x_0$, so $c = x_1-x_0^2 = -1-x_0-x_0^2$ and $x_2 = x_1^2+c = (-1-x_0)^2-1-x_0-x_0^2 = x_0$, again a contradiction. Hence $u \notin \{0,-1\}$.

  By \cref{lem:telescoping}, $(x_0+x_1)(x_1+x_2)(x_2+x_0) = 1$. Since $u \neq 0$, the identities $x_1 = u-x_0$ and $x_2 = v-x_0$ give $v-u = x_2-x_1 = x_1^2-x_0^2 = u(u-2x_0)$. Since $2u \neq 0$, this gives $x_0 = (u^2+u-v)/(2u)$, and hence $x_1+x_2 = u+v-2x_0 = (uv-u+v)/u$. Substituting yields $v(uv-u+v) = 1$, equivalent to
  \begin{equation*}
    (u+1)v^2-uv-1 = (v-1)((u+1)v+1) = 0.
  \end{equation*}
  The roots are $v = 1$ and $v = -1/(1+u)$; the two branches meet at $u = -2$. It remains to rule out $v = 1$ for $u \neq -2$. If $v = 1$, then $x_2 = 1-x_0$ and $x_3 = x_0$ give $c = x_0-(1-x_0)^2$, so $x_1 = x_0^2+c = 3x_0-1$ and $x_2 = x_1^2+c = 8x_0^2-3x_0$. Comparing with $x_2 = 1-x_0$ yields $8x_0^2-2x_0-1 = 0$, whose roots are $x_0 \in \{-\tfrac14,\tfrac12\}$. The case $x_0 = \tfrac12$ gives $x_1 = x_2 = x_0$, contradicting exact period $3$; the case $x_0 = -\tfrac14$ gives $u = -2$, where $1 = -1/(1+u)$. Thus \cref{eq:period3-uv-identity} holds in every case.
\end{proof}

The next proposition gives the full period-$3$ parametrization in the present notation. The formulas for $x_1(t)$ in \cref{eq:x1-period3} and $c(t)$ in \cref{eq:period3-param} go back to Poonen \cite[Theorem~1(3)]{poonen1998} and are recorded by Fu--Stoll \cite[Theorem~2.1(1)]{fu-stoll2024}; an equivalent parametrization of the full $3$-cycle via $u = x_0+x_1$ appears in Erkama \cite[Section~3]{erkama2006}. Poonen's proof goes through general type-classification on dynatomic curves, and Erkama's proof manipulates the period-$3$ identity $(x_0+x_2)(1+(x_0+x_1)) = -1$ in the original $(x,y)$-model. The proof below is shorter and entirely internal to the $(u,v)$-model: the algebraic factorization $V' = U'(1+V-U)$ of \cref{eq:G-factorization} makes the iteration of $G$ on the locus $v = -1/(1+u)$ collapse to a three-step recursion, and the exact-period verification reduces to the rational unsolvability of $u^2+u+1 = 0$. The proof also produces all three cycle points in explicit polynomial form, rather than one point with the others recovered by symmetry.

\begin{proposition} \label{prop:period3-classification}
  For $t \in \Q \setminus \{0,-1\}$, define
  \begin{equation} \label{eq:period3-param}
    c(t) = -\frac{t^6+2t^5+4t^4+8t^3+9t^2+4t+1}{4t^2(t+1)^2}
  \end{equation}
  and
  \begin{align} \label{eq:x0-period3}
    x_0(t) &= \frac{t^3+2t^2+t+1}{2t(t+1)}, \\
    x_1(t) &= \frac{t^3-t-1}{2t(t+1)}, \label{eq:x1-period3} \\
    x_2(t) &= -\frac{t^3+2t^2+3t+1}{2t(t+1)}. \label{eq:x2-period3}
  \end{align}
  Then $x_0(t),x_1(t),x_2(t)$ are distinct and form a rational cycle of exact period $3$ for $f_{c(t)}$. Conversely, if $c \in \Q$ and $f_c$ has a rational point of exact period $3$, then $(c,x_0)$ arises from \cref{eq:period3-param,eq:x0-period3,eq:x1-period3,eq:x2-period3} for a unique $t \in \Q \setminus \{0,-1\}$.
\end{proposition}

\begin{proof}
  Let $t \in \Q \setminus \{0,-1\}$ and set
  \begin{equation*}
    u = t \qquad \text{and} \qquad v = -\frac{1}{1+t}.
  \end{equation*}
  Substituting in \cref{eq:c-in-uv} gives \cref{eq:period3-param}. Substituting in \cref{eq:x0-in-uv,eq:x1-in-uv} and using $x_2 = v-x_0$ gives \cref{eq:x0-period3,eq:x1-period3,eq:x2-period3}. It remains to show exact period $3$. Write $G(U,V) = (U',V')$ for a generic pair with $U \neq 0$. By \cref{eq:G},
  \begin{equation*}
    U' = \frac{-U+V+UV}{U} \qquad \text{and} \qquad V' = \frac{U^2-U+V-U^2V-UV+UV^2+V^2}{U}.
  \end{equation*}
  Expanding the product gives
  \begin{equation*}
    (-U+V+UV)(1+V-U) = U^2-U+V-U^2V-UV+UV^2+V^2.
  \end{equation*}
  Therefore
  \begin{equation} \label{eq:G-factorization}
    V' = U'(1+V-U).
  \end{equation}
  Now iterate from $(u_0,v_0) = (u,v)$. Since $(1+u)v = -1$, we get
  \begin{equation} \label{eq:period3-iterate1}
    u_1 = \frac{-u+v+uv}{u} = \frac{-u+(1+u)v}{u} = -\frac{u+1}{u}, \qquad v_1 = u_1(1+v-u) = u.
  \end{equation}
  Then $v_1 = -1/(1+u_1)$. Repeating the same computation gives
  \begin{equation} \label{eq:period3-iterate2}
    u_2 = -\frac{u_1+1}{u_1} = -\frac{1}{u+1}, \qquad v_2 = -\frac{1}{1+u_2} = -\frac{u+1}{u}.
  \end{equation}
  Since $u \notin \{0,-1\}$, the coordinates $u_0 = u$, $u_1 = -(u+1)/u$, and $u_2 = -1/(u+1)$ are nonzero, so these iterates are defined. One more step gives $u_3 = u$ and $v_3 = v$. Hence $G^3(u,v) = (u,v)$. If $(u,v)$ had period $1$ or $2$ under $G$, then $u_1 = u$ or $u_2 = u$. Each case implies $u^2+u+1 = 0$. This is impossible for $u = t \in \Q$. Hence $(u,v)$ has exact period $3$. By \cref{thm:dictionary}, $x_0(t)$ has exact period $3$ for $f_{c(t)}$. Therefore $x_1(t) = f_{c(t)}(x_0(t))$ and $x_2(t) = f_{c(t)}(x_1(t))$, so $x_0(t),x_1(t),x_2(t)$ are distinct and form an exact period-$3$ cycle.

  Conversely, assume $c \in \Q$ and that $x_0 \in \Q$ has exact period $3$. Define $u = x_0+x_1$ and $v = x_0+x_2$. By \cref{lem:period3-uv},
  \begin{equation*}
    u \in \Q \setminus \{0,-1\} \qquad \text{and} \qquad v = -\frac{1}{1+u}.
  \end{equation*}
  Set $t = u$. Equations \cref{eq:x0-in-uv,eq:c-in-uv} then give $x_0 = x_0(t)$ and $c = c(t)$. The parameter is unique because $t = x_0+x_1$.
\end{proof}

\section{The period-$3$ locus and the ghost curve} \label{sec:period3-ghost}

In this section, we separate the intrinsic period-$3$ dynamics of $G$ from the parametrization of rational period-$3$ cycles. The main point is that the genuine rational period-$3$ locus is the curve $v = -1/(1+u)$, while the remaining algebraic factor is a ghost curve: it appears in the necessary condition $u_0 u_1 u_2 = 1$ but carries no rational period-$3$ orbit.

The relation $v = -1/(1+u)$ obtained in \cref{lem:period3-uv} was extracted from an $f_c$-orbit already assumed to have exact period $3$. We now upgrade the conclusion to a full equivalence at the level of the map $G$ defined in \cref{eq:G}, with the locus described by an intrinsic condition on $(u, v)$ that makes no reference to $f_c$. Resolving the system $G^3(u,v) = (u,v)$ also exposes the geometry behind the relation: the algebraic factorization of $u_0 u_1 u_2 - 1$ splits the locus $u_0 u_1 u_2 = 1$, which away from the fixed line $v = u$ is a necessary condition for period dividing $3$, into the curve $v = -1/(1+u)$ and a single \emph{ghost curve} that satisfies the algebra but carries no rational period-$3$ orbit; see \cref{kuva1}.

\begin{lemma} \label{lem:period3-uv-direct}
  A point $(u,v) \in \Q^2$ has exact period $3$ for $G$ if and only if $u \in \Q \setminus \{0,-1\}$ and
  \begin{equation*}
    v = -\frac{1}{1+u}.
  \end{equation*}
\end{lemma}

\begin{proof}
  Write $G(u,v) = (R(u,v), Q(u,v))$, and for $n \in \{0,1,2,\ldots\}$ let $(u_n, v_n) = G^n(u, v)$, defined whenever the iterates do not pass through $u_k = 0$. The two coordinates of $G$ are tied by the polynomial identity
  \begin{equation} \label{eq:G-factor-direct}
    (-U+V+UV)(1+V-U) = U^2-U+V-U^2V-UV+UV^2+V^2,
  \end{equation}
  which gives $Q(U,V) = R(U,V)(1+V-U)$. In particular,
  \begin{equation} \label{eq:vn-un-recursion}
    v_{n+1}-u_{n+1} = u_{n+1}(v_n-u_n),
  \end{equation}
  so by iteration
  \begin{equation} \label{eq:vn-un-product}
    v_n - u_n = \biggl(\prod_{j=0}^{n-1} u_{j+1}\biggr)(v-u)
  \end{equation}
  whenever $u_1,\ldots,u_n$ are defined and nonzero.

  The forward implication is a direct iteration check that reuses the closed forms already obtained in \cref{prop:period3-classification}. Suppose $u \in \Q \setminus \{0,-1\}$ and $v = -1/(1+u)$. Setting $t = u$ in \cref{prop:period3-classification}, the iteration computed in \cref{eq:period3-iterate1,eq:period3-iterate2} gives $u_1 = -(u+1)/u$, $u_2 = -1/(u+1)$, and $(u_3, v_3) = (u, v)$. The values $u_0, u_1, u_2$ are pairwise distinct, since $u_j = u_k$ for $j \neq k$ in $\{0,1,2\}$ would force $u^2+u+1 = 0$, which has no rational solution. Therefore $(u, v)$ has exact period $3$ for $G$.

  For the converse, rather than analysing the system $G^3(u,v) = (u,v)$ directly, we transport the orbit through \cref{thm:dictionary} to the $f_c$-side, where \cref{lem:period3-uv} already pins down the relation between $u$ and $v$. Suppose $(u, v) \in \Q^2$ has exact period $3$ for $G$. Then $G$ is defined at $(u_0, v_0)$, $(u_1, v_1)$, $(u_2, v_2)$, so $u_0, u_1, u_2 \neq 0$. Define $x_0$, $x_1$, and $c$ by \cref{eq:x0-in-uv,eq:x1-in-uv,eq:c-in-uv}. Then $x_0+x_1 = u$, the definition of $c$ gives $x_1 = f_c(x_0)$, and a direct substitution gives $x_0+f_c(x_1) = v$. By \cref{thm:dictionary}, the $f_c$-orbit $(x_n)$ reconstructed from $(u, v)$ has exact period $3$. \cref{lem:period3-uv} then gives $u \neq -1$ and $v = -1/(1+u)$.

  \begin{figure}[t]
    \begin{center}
    \includegraphics[width=0.6\textwidth,angle=0]{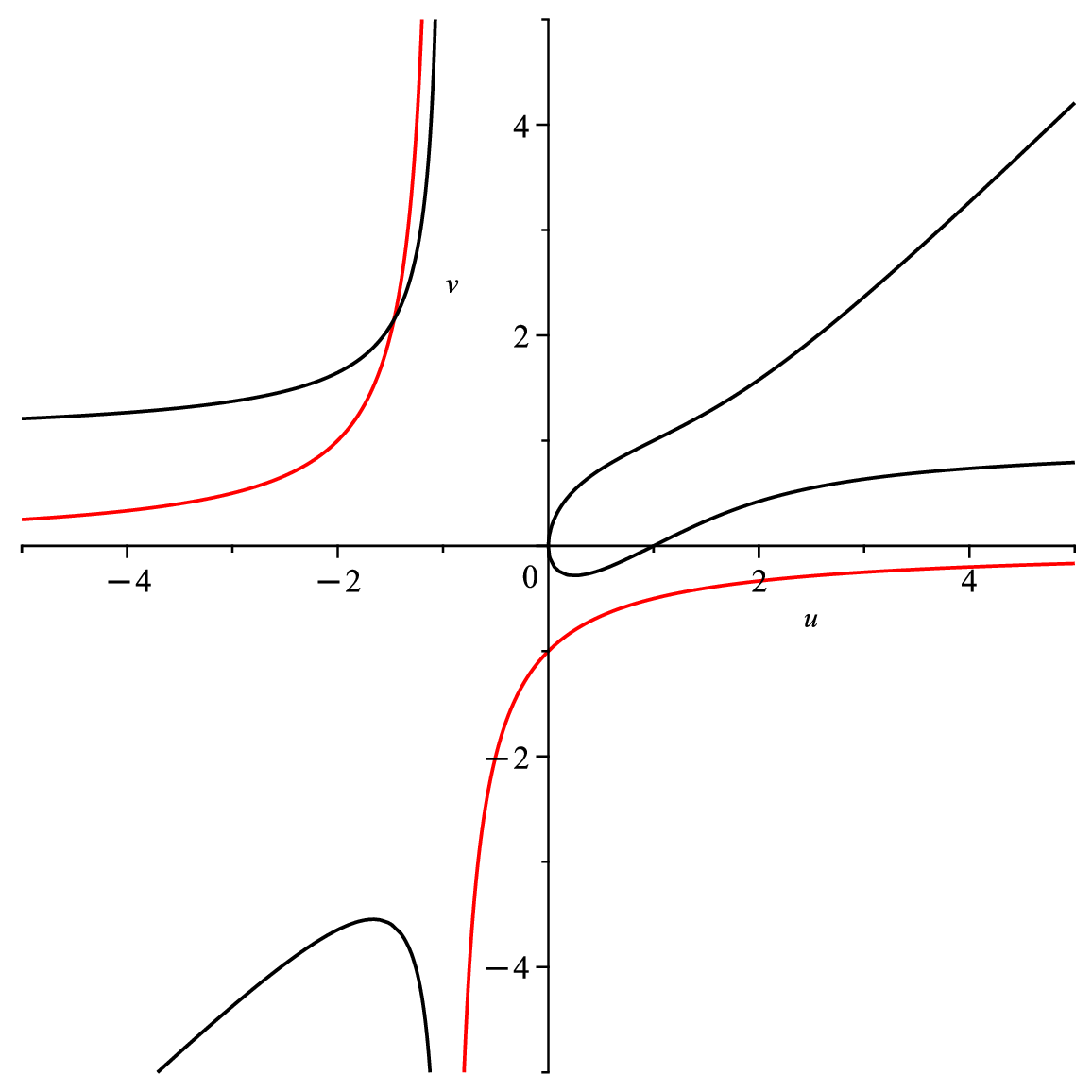}
    \caption{The period-$3$ curve $v = -1/(1+u)$ and the ghost curve $u^2 v - u^2 - u v^2 + uv + u - v^2 = 0$ on the real $(u,v)$-plane.} \label{kuva1}
    \end{center}
  \end{figure}

  We close the proof with the geometric picture underlying \cref{kuva1}. Applying \cref{eq:vn-un-product} with $n = 3$ and $(u_3, v_3) = (u, v)$ gives
  \begin{equation*}
    (v-u)(u_0 u_1 u_2 - 1) = 0.
  \end{equation*}
  Exact period $3$ rules out $v = u$, since the line $v = u$ is fixed by $G$ as $R(u,u) = Q(u,u) = u$, so we obtain $u_0 u_1 u_2 = 1$. Expanding this with the closed forms $u_1 = (-u+v+uv)/u$ and $u_2 = (-u_1+v_1+u_1v_1)/u_1$ in terms of $(u,v)$ yields the polynomial factorization
  \begin{equation} \label{eq:ghost-factorization}
    u_0 u_1 u_2 - 1 = -\frac{(uv+v+1)(u^2 v - u^2 - u v^2 + uv + u - v^2)}{u}.
  \end{equation}
  Its zero locus is the union of the period-$3$ curve $uv + v + 1 = 0$, equivalent to $v = -1/(1+u)$, and the \emph{ghost curve}
  \begin{equation*}
    u^2 v - u^2 - u v^2 + uv + u - v^2 = 0.
  \end{equation*}
  By the necessity argument above, every rational period-$3$ point lies on $v = -1/(1+u)$, so the ghost curve carries no rational period-$3$ orbit. Substituting $v = -1/(1+u)$ into the ghost equation and clearing denominators gives $-(u+1)(u^3+u^2+1) = 0$; excluding $u = -1$, the rational intersection points would satisfy $u^3 + u^2 + 1 = 0$, which has no rational root, since by the rational root theorem any such root would be $\pm 1$ and neither is a root. Hence the ghost curve and the period-$3$ curve do not meet rationally.
\end{proof}

Numerically, the cubic $u^3+u^2+1 = 0$ identified in the proof of \cref{lem:period3-uv-direct} has one real root $u \approx -1.465571232$; the corresponding point $(u, v)$ with $v \approx 2.147899036$ on the period-$3$ curve is the visible intersection of the two curves in \cref{kuva1}. The other two roots are the complex conjugate pair $u \approx 0.2327856159 \pm 0.7925519925i$.

The ghost factor in \cref{eq:ghost-factorization} nevertheless defines a genuine cubic curve. Call a projective plane curve \emph{geometrically irreducible} if it is irreducible over the algebraic closure. A geometrically irreducible projective plane cubic with a rational singular point is rational over $\Q$: the singular point of such a cubic is necessarily a double point, and projection from it parametrizes the curve; see Silverman \cite[Proposition~III.1.6]{silvermanAEC} for the corresponding statement in Weierstrass form. The projective closure of the ghost curve is instead nonsingular and therefore geometrically irreducible, since by B\'ezout's theorem distinct components of a projective plane curve meet over the algebraic closure and every point where two components meet is singular; see Fulton \cite[\S 5.3 and Problem~5.21]{fulton2008}. This distinction does not itself explain the absence of rational period-$3$ orbits, since a nonsingular cubic with a rational point can still have infinitely many rational points; that absence follows from the closing argument in the proof of \cref{lem:period3-uv-direct}.

Since the closing argument in the proof of \cref{lem:period3-uv-direct} already rules out rational period-$3$ orbits on the ghost curve, the next proposition is not needed for our main results; it identifies the projective closure of the ghost curve as an elliptic curve and explains why the three disjoint affine arcs in \cref{kuva1} belong to one connected real projective curve. A nonsingular projective plane cubic has genus $1$ by the degree-genus formula and therefore admits no rational parametrization, since the genus of the target of a nonconstant map of smooth curves is at most the genus of the source; see \cite[Exercises~II.2.7 and~II.2.8(a)]{silvermanAEC}. Such a curve with a rational point is an elliptic curve over $\Q$ once that point is taken as the origin; see \cite[\S III.3, Definition and Proposition~III.3.1]{silvermanAEC}. The standard Weierstrass equation, its discriminant, and the criterion that the curve is nonsingular exactly when the discriminant is nonzero are in \cite[\S III.1 and Proposition~III.1.4]{silvermanAEC}.

The reduction of a plane cubic to a Weierstrass equation is classical; a Weierstrass model is obtained in \cite[Proposition~III.3.1]{silvermanAEC} by means of the Riemann--Roch theorem \cite[Theorem~II.5.4 and Corollary~II.5.5]{silvermanAEC}. To produce an explicit change of coordinates for the ghost curve, we use an elementary construction in homogeneous coordinates $[u:v:w]$ on the projective plane, where $(u, v, w) \neq (0, 0, 0)$, nonzero scalar multiples represent the same point, and $w = 1$ recovers the affine $(u, v)$-plane. We first find a rational point at which the tangent line has contact of order three, choose it as the origin, and then apply a projective linear change of coordinates carrying the point to $[0:1:0]$ and the tangent to the line at infinity. Completing the square and rescaling then put the transformed cubic into the desired Weierstrass form. \cref{prop:ghost-elliptic} carries this out for the ghost curve, exhibiting the point, the change of coordinates, and the resulting equation.

\begin{proposition} \label{prop:ghost-elliptic}
  The projective closure of the ghost curve
  \begin{equation} \label{eq:ghost-curve}
    u^2 v - u^2 - u v^2 + uv + u - v^2 = 0
  \end{equation}
  obtained by adding its three points at infinity is a nonsingular plane cubic over $\Q$. The tangent at $(1, 0)$ meets the curve with third-order contact, so this rational point can serve as the point at infinity in a Weierstrass equation. More explicitly, a projective linear change of coordinates given in the affine coordinates $(u, v)$ and $(X, Y)$ by
  \begin{equation} \label{eq:ghost-weierstrass-map}
    X = \frac{8v}{u - 2v - 1} \qquad \text{and} \qquad Y = \frac{8(u - v + 1)}{u - 2v - 1}
  \end{equation}
  transforms the projective closure of the ghost curve into the projective closure of
  \begin{equation} \label{eq:ghost-weierstrass}
    Y^2 = X^3 + 13 X^2 + 48 X + 64,
  \end{equation}
  which is nonsingular. Thus the projective closure of the ghost curve, with the rational point $(1, 0)$ as origin, is an elliptic curve over $\Q$. Over $\R$, its real locus is connected; removing the three points at infinity leaves the three arcs visible in the affine $(u, v)$-plane.
\end{proposition}

\begin{proof}
  Homogenizing the ghost curve, as in \cite[\S I.2 and Example~I.2.8]{silvermanAEC}, gives $F(u, v, w) = u^2 v - u^2 w - u v^2 + uvw + uw^2 - v^2 w$. Its projective closure is $F = 0$, and $w = 0$ is the line at infinity. A point of the projective closure is nonsingular exactly when the three partial derivatives of $F$ do not all vanish there, so a single nonvanishing partial derivative already shows that the point is nonsingular. Since $F$ is homogeneous of degree $3$, its partial derivatives are homogeneous of degree $2$, so rescaling a representative by a nonzero scalar multiplies each of them by the square of that scalar. Their values thus depend on the representative, while their simultaneous vanishing does not; see \cite[Problem~5.1]{fulton2008}. Differentiating $F$ gives
  \begin{equation} \label{eq:ghost-partials}
    \begin{split}
      \frac{\partial F}{\partial u} &= 2uv - 2uw - v^2 + vw + w^2, \\
      \frac{\partial F}{\partial v} &= u^2 - 2uv + uw - 2vw, \\
      \frac{\partial F}{\partial w} &= -u^2 + uv + 2uw - v^2.
    \end{split}
  \end{equation}
  Setting $w = 0$ in $F$ yields $uv(u - v) = 0$, so the line at infinity meets the projective closure in the three distinct points $[1:0:0]$, $[0:1:0]$, and $[1:1:0]$. On that line $\partial F/\partial w$ restricts to $-(u^2 - uv + v^2)$, which equals $-1$ at each of the representatives $(1, 0, 0)$, $(0, 1, 0)$, and $(1, 1, 0)$, so no point at infinity is singular. Two of the other partial derivatives do vanish at points of the line at infinity, namely $\partial F/\partial u$ at $[1:0:0]$ and $\partial F/\partial v$ at $[0:1:0]$, but only the simultaneous vanishing of all three marks a singular point. The restriction $F(u, v, 0) = uv(u - v)$ has three simple roots, and on $w = 0$ the derivatives $\partial F/\partial u$ and $\partial F/\partial v$ restrict to the two partial derivatives of this binary cubic, so a singular point at infinity would be a multiple root of $uv(u - v)$; the line at infinity meets the curve transversally. Setting $w = 1$ in \cref{eq:ghost-partials}, the system $\partial F/\partial u = 2uv - 2u - v^2 + v + 1 = 0$ and $\partial F/\partial v = u^2 - 2uv + u - 2v = 0$ has no solution on the affine curve. Indeed, the line $u = -1$ is disjoint from the curve since $F(-1, v, 1) = -2$. For $u \neq -1$, the second equation gives $v = u/2$; substituting this into the first equation gives $3u^2 - 6u + 4 = 0$, while substituting it into the curve equation gives $u(u^2 - 3u + 4) = 0$, in both cases after clearing a common factor $1/4$. These two equations have no common solution, even over the algebraic closure, since the resultant of $3u^2 - 6u + 4$ and $u(u^2 - 3u + 4)$ equals $112 \neq 0$. Hence the projective closure of the ghost curve is nonsingular. The point $(1, 0)$ lies on the ghost curve, since the left-hand side of \cref{eq:ghost-curve} vanishes there, and the partial derivatives at $(1, 0)$ are $-1$ and $2$, giving tangent line $u = 2v + 1$; substituting into the left-hand side of \cref{eq:ghost-curve} yields $2v^3$, so the tangent has third-order contact at $(1, 0)$.

  The linear change of homogeneous coordinates $(\hat X, \hat Y, \hat Z) = (v, u, u - 2v - w)$ sends $(1, 0, 1)$ and its tangent line to the point and line at infinity in the new coordinates. In these coordinates $F$ becomes $2\hat X^3 + \hat X^2 \hat Z + 3\hat X \hat Y \hat Z - \hat Y^2 \hat Z + \hat Y \hat Z^2$. Dehomogenizing with respect to $\hat Z$ and completing the square in $\hat Y$ yields
  \begin{equation*}
    (2\hat Y - 3\hat X - 1)^2 = 8\hat X^3 + 13\hat X^2 + 6\hat X + 1,
  \end{equation*}
  and the rescaling $X = 8\hat X$, $Y = 8(2\hat Y - 3\hat X - 1)$ produces \cref{eq:ghost-weierstrass}. In homogeneous coordinates, the combined transformation is $[u:v:w] \mapsto [8v:8(u - v + w):u - 2v - w]$. Its matrix has determinant $128$, so it is projective linear, and dehomogenizing it in the affine coordinates used in \cref{eq:ghost-weierstrass-map} gives the displayed rational formulas. The discriminant of \cref{eq:ghost-weierstrass}, formed as in \cite[\S III.1]{silvermanAEC}, equals $-114688$, which is nonzero, so \cref{eq:ghost-weierstrass} is nonsingular; see \cite[Proposition~III.1.4]{silvermanAEC}.

  Finally, the cubic on the right-hand side of \cref{eq:ghost-weierstrass} factors as $(X + 8)(X^2 + 5X + 8)$, and the quadratic factor has discriminant $-7$, so $X = -8$ is the only real root and the right-hand side is nonnegative exactly when $X \ge -8$. The real affine locus of \cref{eq:ghost-weierstrass} is therefore the union of the two graphs $Y = \pm\sqrt{(X + 8)(X^2 + 5X + 8)}$ over $[-8, \infty)$, which meet at $(-8, 0)$. Since a Weierstrass curve has the single point $[0:1:0]$ at infinity by \cite[\S III.1]{silvermanAEC} and both unbounded ends of the affine locus approach it, the real locus of the projective closure of \cref{eq:ghost-weierstrass} is a single closed curve, a topological circle. The transformation above is projective linear with real entries, so it induces a homeomorphism of the real projective plane carrying the real locus of the projective closure of the ghost curve onto this circle; in particular, the former is connected. Under the transformation, the points $[1:0:0]$, $[0:1:0]$, and $[1:1:0]$ at infinity map respectively to $(0, 8)$, $(-4, 4)$, and $(-8, 0)$. Removing these three distinct points from the circle leaves exactly three connected arcs, whose preimages are the connected components of the ghost curve in the affine $(u, v)$-plane, proving the final assertion.
\end{proof}

\section{Arithmetic of rational period-$3$ cycles} \label{sec:period3-arith}

With the period-$3$ cycles described, we prove the main theorem, \cref{thm:period3}, together with two supporting results. The first is elementary: \cref{prop:fixed-z-period3-cubic} shows that for each fixed $z \in \Q$, at most three parameters $c$ place $z$ on a rational period-$3$ cycle. The second is arithmetic: the denominator theorem \cref{prop:ef-transfer} feeds into \cref{prop:period3-rational-sunit}, which makes the period parameter $t$ and the shifted value $t+1$ into $S$-units. Applying \cref{thm:abc-template} to this $S$-unit structure then bounds the height of $t$ in terms of the primes dividing $\den(c)$ under the ABC conjecture, which is \cref{thm:period3}.

We now fix $x_0 = z$ and let $c$ vary. In the $(u,v)$-plane, this restriction is a rational parabola; intersecting it with the period-$3$ relation $v = -1/(1+u)$ of \cref{lem:period3-uv} gives a cubic, equivalent to $x_0(t) = z$ in the parametrization of \cref{prop:period3-classification}. The upper bound ``at most three $c$'' is the cubic degree in $c$ of the period-$3$ dynatomic polynomial; see Silverman \cite[Remark~4.3 and Table~4.1]{silverman2007}. This is the direct mechanism behind \cref{prop:fixed-z-period3-cubic}.

\begin{proposition} \label{prop:fixed-z-period3-cubic}
  Let $z,c \in \Q$. Then $z$ has exact period $3$ for $f_c$ if and only if
  \begin{equation} \label{eq:fixed-z-period3-cubic}
    u^3+2(1-z)u^2+(1-2z)u+1 = 0
  \end{equation}
  with $u = z^2+z+c$. In particular, for each fixed $z \in \Q$, there are at most three parameters $c \in \Q$ such that $z$ has exact period $3$ for $f_c$.
\end{proposition}

\begin{proof}
  Set $u = x_0+x_1 = z^2+z+c$, so $c = u-z-z^2$. By \cref{prop:period3-classification}, $z$ has exact period $3$ for $f_c$ if and only if $(c, z) = (c(t), x_0(t))$ for a unique $t \in \Q \setminus \{0,-1\}$. From \cref{eq:x0-period3,eq:x1-period3} we read $x_0(t)+x_1(t) = t$, forcing $t = u$. The condition $x_0(u) = z$ then reads
  \begin{equation*}
    \frac{u^3+2u^2+u+1}{2u(u+1)} = z,
  \end{equation*}
  which clears to \cref{eq:fixed-z-period3-cubic}. Conversely, suppose \cref{eq:fixed-z-period3-cubic} holds for $u = z^2+z+c$. Neither $u = 0$ nor $u = -1$ is a root, so $u \in \Q \setminus \{0,-1\}$, and reversing the clearing gives $x_0(u) = z$. Since $x_0(u)+x_1(u) = u$, we have $x_1(u) = u-z$, and hence $c(u) = x_1(u)-x_0(u)^2 = u-z-z^2 = c$. Thus \cref{prop:period3-classification} shows that $z$ has exact period $3$ for $f_c$. The relation $c = u-z-z^2$ is affine in $u$, so each rational root determines a unique $c$; hence at most three.
\end{proof}

The next proposition reformulates the denominator structure of a rational period-$3$ cycle as an $S$-unit equation. The arithmetic content is not new: Eliahou--Fares \cite[Corollary~2.6]{eliahou-fares2022} prove that, for any two distinct periodic points $x \neq y$ of $\varphi_c$ with respective integer numerators $X$ and $Y$ (when written over the common denominator $d$), every prime divisor of $X+Y$ already divides $d$; applied to a period-$3$ cycle, this places both $X_0+X_1$ and $X_0+X_2$ into $\Z[S^{-1}]^\times$. Equivalently, Morton--Raianu \cite[Theorem~2(a)]{mortonraianu2022} compute $\den(c) = 16(mn(m+n)/2)^2$ for $t = m/n$ in lowest terms, which encodes the same constraint by listing the prime factors of $d$ explicitly. We do not reprove these facts. What \cref{prop:period3-rational-sunit} adds is the formulation \cref{eq:period3-rational-sunit} of the constraint as the $S$-unit equation $t/(t+1)+1/(t+1) = 1$, which is the form required for the proof of \cref{thm:period3}. The short proof below uses the period-$3$ identity \cref{lem:period3-uv} to convert the Eliahou--Fares conclusion for the pair $(x_0,x_2)$ into a statement about $t+1$, and so works directly with $t$ and $v$ rather than through the explicit parametrization of \cref{prop:period3-classification}.

\begin{proposition} \label{prop:period3-rational-sunit}
  Let $c \in \Q$ and assume that $f_c$ has a rational cycle of exact period $3$. Write $\den(c) = d^2$ and $S = \supp(d)$. If $t$ is the parameter from \cref{prop:period3-classification}, then
  \begin{equation} \label{eq:period3-sunits}
    t,t+1 \in \Z[S^{-1}]^\times.
  \end{equation}
  In particular,
  \begin{equation} \label{eq:period3-rational-sunit}
    \frac{t}{t+1}+\frac{1}{t+1} = 1
  \end{equation}
  is an $S$-unit equation.
\end{proposition}

\begin{proof}
  Let $(x_0,x_1,x_2)$ be the cycle, indexed so that the parameter $t$ from \cref{prop:period3-classification} is attached to $x_0$. By \cref{prop:ef-transfer}, each $x_j$ has denominator $d$. Hence
  \begin{equation*}
    x_j = \frac{X_j}{d}
  \end{equation*}
  with $X_j \in \Z$ and $\gcd(X_j,d) = 1$ for $j \in \{0,1,2\}$. By \cref{lem:telescoping}, $(x_0+x_1)(x_1+x_2)(x_2+x_0) = 1$. Substituting $x_j = X_j/d$ gives
  \begin{equation*}
    (X_0+X_1)(X_1+X_2)(X_2+X_0) = d^3.
  \end{equation*}
  Hence every prime divisor of each sum $X_i+X_j$ lies in $S$.

  By the converse statement in \cref{prop:period3-classification}, the period parameter is
  \begin{equation*}
    t = x_0+x_1.
  \end{equation*}
  Set $v = x_0+x_2$. Since
  \begin{equation*}
    (x_0+x_1)(x_1+x_2)(x_2+x_0) = 1,
  \end{equation*}
  we have $x_0+x_1 \neq 0$ and $x_0+x_2 \neq 0$. Equivalently,
  \begin{equation*}
    X_0+X_1 \neq 0 \qquad \text{and} \qquad X_0+X_2 \neq 0.
  \end{equation*}
  Therefore
  \begin{equation*}
    t = \frac{X_0+X_1}{d} \qquad \text{and} \qquad v = \frac{X_0+X_2}{d},
  \end{equation*}
  so $t,v \in \Z[S^{-1}]^\times$. By \cref{lem:period3-uv}, $v = -1/(1+t)$. Thus $t+1 = -1/v \in \Z[S^{-1}]^\times$. This gives \cref{eq:period3-sunits}, and \cref{eq:period3-rational-sunit} is immediate.
\end{proof}

We are now ready to prove the main theorem. The argument turns the $S$-unit relation of \cref{prop:period3-rational-sunit} into a quantitative estimate, depending only on that relation and \cref{conj:abc}: since $t$ and $t+1$ are $S$-units, the bound follows by applying the template \cref{thm:abc-template} to the associated coprime triples.

\begin{proof}[Proof of \cref{thm:period3}]
  By \cref{prop:period3-rational-sunit}, applied to the period parameter $t = x_0+x_1$, the quantity $\tau_3 = t$ satisfies \cref{eq:periodn-tau-sunits}. Applying \cref{thm:abc-template} with $n = 3$ and $\tau_3 = t$ gives
  \begin{equation*}
    \max\{|a|,|b|,|a+b|\} \le C_\eps \biggl( \prod_{p\in S} p \biggr)^{1+\eps}
  \end{equation*}
  and
  \begin{equation*}
    h(t) \le (1+\eps)\sum_{p\in S}\log p + O_\eps(1).
  \end{equation*}
  These are exactly \cref{eq:period3-abc,eq:abc-log-height-bound-t}.
\end{proof}

The height bound in \cref{thm:period3} also gives a simple count for fixed denominator support. Write $S = \{p_1,\ldots,p_s\}$. Each $S$-unit has the form $\pm p_1^{e_1}\cdots p_s^{e_s}$ with $e_i \in \Z$, and the condition $h(r) \le H$ implies $|e_i|\log p_i \le H$ for every $i$. Since $2\lfloor H/\log p \rfloor+1 \le (1+2/\log p)(1+H) \le 4(1+H)$ for every prime $p$ and every $H \ge 0$, we obtain
\begin{equation*}
  \#\{r \in \Q : h(r) \le H \text{ and } r \in \Z[S^{-1}]^\times\} \le 2\prod_{p \in S} \biggl(2\biggl\lfloor \frac{H}{\log p} \biggr\rfloor+1\biggr) \le 2 \cdot 4^{\# S}(1+H)^{\# S}.
\end{equation*}
By \cref{eq:period3-abc}, the period parameter of a rational period-$3$ cycle satisfies $h(t) \le \log C_\eps+(1+\eps)\log \rad(d)$, where $\rad(d) = \prod_{p \in S} p$ and $C_\eps$ is the constant of \cref{conj:abc}, enlarged if necessary so that $C_\eps \ge 1$. Taking $H$ to be this bound gives, under \cref{conj:abc},
\begin{align*}
  \#\{c \in \Q : f_c \text{ has a rational } &\text{period-$3$ cycle with } \supp(d) = S\} \\
  &\le 2 \cdot 4^{\# S}\biggl(1+\log C_\eps+(1+\eps)\log \prod_{p \in S} p\biggr)^{\# S}
\end{align*}
for every $\eps > 0$. The count runs over the period parameter $t$ rather than over $c$ directly, which is legitimate because \cref{prop:period3-classification} produces for every such $c$ a parameter $t$ with $c = c(t)$, while \cref{prop:period3-rational-sunit} makes $t$ an $S$-unit and \cref{thm:period3} bounds its height. The two heights are comparable in any case, since $c(t)$ has degree $6$ as a rational function of $t$ and hence $h(c(t)) = 6h(t)+O(1)$. Thus, after the denominator support has been fixed, the conditional height bound leaves only polynomially many admissible $S$-unit parameters, with exponent $\# S$.

We now prove the two unconditional companions to \cref{thm:period3} stated in \cref{sec:main}. Both follow from the same $S$-unit identity \cref{eq:period3-rational-sunit} by replacing \cref{conj:abc} with classical Diophantine machinery. The first, \cref{prop:period3-evertse-count}, bounds the number of admissible parameters $c$ by a function of $\# S$ alone, via Evertse's bound for the $S$-unit equation.

\begin{proof}[Proof of \cref{prop:period3-evertse-count}]
  By \cref{prop:period3-classification}, every such $c$ equals $c(t)$ for some $t \in \Q \setminus \{0, -1\}$, and distinct values of $c$ require distinct $t$; fix one such $t$ for each $c$. By \cref{prop:period3-rational-sunit}, applied with $S_0 = \supp(d)$, we have $t, t+1 \in \Z[S_0^{-1}]^\times \subseteq \Z[S^{-1}]^\times$. Setting $x = 1/(t+1)$ and $y = t/(t+1)$ then gives a solution of
  \begin{equation*}
    x + y = 1
  \end{equation*}
  with $x, y \in \Z[S^{-1}]^\times$, and the pair $(x, y)$ determines $t = y/x$. By Evertse \cite[Theorem~1]{evertse1984}, the number of solutions $(x, y)$ in $S$-units (in the sense of Evertse, with $S$ extended to include the infinite prime, so that $s = \# S + 1$ and the degree parameter is $m = 1$, his field $K$ being $\Q$) is at most $3 \cdot 7^{m + 2s} = 3 \cdot 7^{2\# S + 3}$.
\end{proof}

The second, \cref{prop:period3-unconditional-height}, gives a weak height bound on $t$ via the best known unconditional ABC estimate. 

\begin{proof}[Proof of \cref{prop:period3-unconditional-height}]
  By \cref{prop:period3-rational-sunit}, $t = a/b$ with $\gcd(a, b) = 1$ and all prime divisors of $a$, $b$, $a+b$ lying in $S$, so $\rad(ab(a+b)) \le \prod_{p \in S} p$. The integers $a$, $b$, and $a+b$ are pairwise coprime and nonzero; after reordering according to signs, $|a|$, $|b|$, and $|a+b|$ form an ABC triple in the sense of Stewart--Yu. Applying the unconditional ABC estimate of Stewart--Yu \cite[Theorem~1]{stewart-yu2001} gives
  \begin{align*}
    \max\{|a|, |b|, |a+b|\} &< \exp(C_0 \rad(ab(a+b))^{1/3} (\log \rad(ab(a+b)))^3) \\
    &\le \exp\biggl(C_0 \biggl(\prod_{p \in S} p\biggr)^{1/3} \biggl(\log \prod_{p \in S} p\biggr)^3\biggr).
  \end{align*}
  Taking logarithms and using $h(t) = \log \max\{|a|, |b|\}$ gives \cref{eq:stewart-yu-height-bound-t}.
\end{proof}

\printbibliography

\end{document}